\documentclass[reqno,12pt]{amsart}
\usepackage{amsmath,amsfonts,amsthm,amssymb,indentfirst}
\usepackage[all,poly]{xy} 
\usepackage{color}
\usepackage[pagebackref,colorlinks]{hyperref}

\usepackage{mathrsfs}
\usepackage{enumerate}

\theoremstyle{plain}
\newtheorem{lemma}{Lemma}[section] 
\newtheorem{theorem}[lemma]{Theorem}
\newtheorem{corollary}[lemma]{Corollary}
\newtheorem{proposition}[lemma]{Proposition}

\theoremstyle{definition}

\newtheorem{example}[lemma]{Example}
\newtheorem{definition}[lemma]{Definition}
\newtheorem{remarks}[lemma]{Remarks}

\newcommand{\Zset}{\mathbb Z}
\newcommand{\Nset}{\mathbb N}

\newcommand{\M}{\operatorname{\mathbb M}}
\newcommand{\gr}{\operatorname{gr}}
\newcommand{\ol}{\overline}

\newcommand{\so}{\mathbf{s}}
\newcommand{\ra}{\mathbf{r}}

\newcommand{\G}{\mathcal G}

\title[Graded chain conditions and Leavitt path algebras of no-exit graphs]{Graded chain conditions and Leavitt path algebras of no-exit graphs}

\author{Lia Va\v s}
\address{Department of Mathematics, Physics and Statistics, University of the Sciences, Philadelphia, PA 19104, USA}
\email{l.vas@usciences.edu}



\thanks{The author is grateful to Roozbeh Hazrat for his insightful comments on graded rings and to the referee for thoughtful suggestions and inspiring remarks.}

\begin{document}
\begin{abstract} We obtain a complete structural characterization of Cohn-Leavitt algebras over no-exit objects as graded involutive algebras. Corollaries of this result include graph-theoretic conditions characterizing when a Leavitt path algebra is a directed union of (graded) matricial algebras over the underlying field and over the algebra of Laurent polynomials and when the monoid of isomorphism classes of finitely generated projective modules is atomic and cancellative. 

We introduce the non-unital generalizations of graded analogues of noetherian and artinian rings, graded locally noetherian and graded locally artinian rings, and characterize graded locally noetherian and graded locally artinian Leavitt path algebras without any restriction on the cardinality of the graph. As a consequence, we relax the assumptions of the Abrams-Aranda-Perera-Siles characterization of locally noetherian and locally artinian Leavitt path algebras. 
\end{abstract}

\maketitle
 
\section{Introduction}

Leavitt path algebras of directed graphs are the algebraic counterparts of graph $C^*$-algebras and generalizations of Leavitt algebras. Since the introduction of Leavitt path algebras in the mid 2000s, many results, generalizations, and applications to other areas of mathematics have been obtained and further directions of research developed. In the first several years after their introduction, Leavitt path algebras were considered primarily for countable graphs and, in many cases, just for row-finite graphs, i.e. for graphs with vertices that emit only finitely many edges. Subsequently, some results, previously shown for countable and row-finite graphs, were shown to hold for arbitrary graphs. In this paper, we continue this trend.

If $K$ is a field, and $E$ a row-finite and countable graph in which no cycle has an exit (no-exit graph) and in which infinite paths end in sinks or cycles, then 
the Leavitt path algebra $L_K(E)$ is isomorphic to a direct sum of locally matricial algebras of the form $\M_\kappa(K)$ and $\M_\mu(K[x,x^{-1}])$ where $\kappa$ and $\mu$ are countable cardinals by \cite[Theorem 3.7]{AAPM}. More recently, in \cite{Roozbeh_Ranga}, the authors show that if $E$ is as above but not necessarily countable, then $L_K(E),$ naturally graded by the group of integers $\Zset,$ is graded isomorphic to a direct sum of graded locally matricial algebras of the form $\M_\kappa(K)$ and $\M_\mu(K[x^n,x^{-n}])$ where $\kappa$ and $\mu$ are cardinals and $n$ positive integers. Finally, in \cite{Gonzalo_Mercedes_Brox}, the Leavitt path algebras over arbitrary no-exit graphs have been characterized as directed unions of direct sums of (non-graded) matricial algebras over $K$ and $K[x,x^{-1}].$ We provide a complete characterization of Leavitt path algebras of no-exit graphs as {\em graded} and {\em involutive} algebras without any restriction on the cardinality of the graph or the number of edges which vertices emit (Proposition \ref{ultramatricial_rep}). Therefore, our result is an involutive and graded version of \cite[Theorem 3.2, equivalences (i) and (ii)]{Gonzalo_Mercedes_Brox} and a generalization of \cite[Theorem 6.7, equivalences (2) and (3)]{Roozbeh_Ranga}. 

Moreover, we show that the no-exit graphs are the {\em only} graphs with the Leavitt path algebras (graded) $*$-isomorphic to directed unions of algebras of the form $A_i\oplus B_i$ for some directed set $I,$ (graded) matricial $*$-algebras $A_i$ over $K,$ and (graded) matricial $*$-algebras $B_i$ over $K[x^{n_i},x^{-n_i}]$ for positive integers $n_i$ and $i\in I$ (Corollary \ref{no-exit_corollary}). The Leavitt path algebras of acyclic graphs are the {\em only} Leavitt path algebras which are (graded) $*$-isomorphic to a directed union of (graded) matricial $*$-algebras over $K$ (Corollary \ref{acyclic_corollary}). In addition, the Leavitt path algebras of row-finite, no-exit graphs in which infinite paths end in sinks or cycles are the {\em only} Leavitt path algebras which are (graded) $*$-isomorphic to a direct sum of the (graded) matrix $*$-algebras of the form $\M_\kappa(K)$ and $\M_\mu(K[x^n,x^{-n}])$ where $\kappa$ and $\mu$ are cardinals and $n$ positive integers, and are the {\em only} Leavitt path algebras whose monoid of the isomorphism classes of finitely generated projective modules is atomic and cancellative (Corollaries \ref{no-exit_corollary} and \ref{projectives}). We also show that, surprisingly, not every graded matrix $*$-algebra is graded isomorphic to a Leavitt path algebra (Proposition \ref{realization_question}).

We formulate our results for Cohn-Leavitt algebras (of non-separated graphs) considered in \cite{Ara_Goodearl}. Although this class of algebras is not any larger than the class of Leavitt path algebras, using Cohn-Leavitt algebras has one important advantage over using Leavitt path algebras alone -- a Cohn-Leavitt algebra can be represented as a directed union of the subalgebras over the finite, complete subgraphs. We make use of this fact throughout section \ref{section_direct_limits} and present examples which further illustrate the benefits of using Cohn-Leavitt algebras over using Leavitt path algebras alone.

A prominent direction of research on Leavitt path algebras is the characterization of the ring-theoretic properties
of a Leavitt path algebra $L_K(E)$ in terms of the graph-theoretic properties of the graph $E$, i.e. results of the form 
\begin{center}
{\em $L_K(E)$ has ring-theoretic property $(P)$ if and only if  $E$ has graph-theoretic property $(P').$ }
\end{center}
While relevant in its own right, this line of research has also become a way to create rings with various predetermined properties by selecting suitable graphs. We focus on properties of being noetherian and artinian and their graded generalizations. It is known that $L_K(E)$ is noetherian if and only if $E$ is a finite, no-exit graph and $L_K(E)$ is artinian if and only if $E$ is a finite, acyclic graph. While it is known that $L_K(E)$ is neither noetherian or artinian unless $E$ is finite, it is still of interest to characterize when $L_K(E)$ is ``locally'' noetherian and artinian. With these generalizations appropriately defined, locally noetherian and locally artinian Leavitt path algebras of countable, row-finite graphs have been characterized in  \cite[Theorems 2.4 and 3.7]{AAPM}.

If the graded structure of the algebra is taken into account, the characterization results have the form 
\begin{center}
{\em $L_K(E)$ has graded ring-theoretic property $(P)$ if and only if  $E$ has graph-theoretic property $(P').$ }
\end{center}
Recently, many ring-theoretic properties have been adapted to graded rings, for example, graded von Neumann regular in \cite{Roozbeh_regular},  graded directly finite in \cite{Roozbeh_Lia} and graded Baer in \cite{Roozbeh_Lia_Baer}. We introduce the graded versions of the non-unital generalizations of noetherian and artinian rings, graded locally noetherian and graded locally artinian rings. In Theorems \ref{noetherian} and \ref{artinian}, we characterize graded locally noetherian and graded locally artinian Leavitt path algebras. We also show that the assumptions that the underlying graph is countable and row-finite can be dropped from the characterization of locally noetherian and artinian Leavitt path algebras (\cite[Theorems 2.4 and 3.7]{AAPM}). Our results imply that the property of being locally noetherian is invariant for the graded structure of Leavitt path algebras, while the property of being locally artinian is not, i.e. the following holds for a Leavitt path algebra $L_K(E)$. 
\begin{center}
\begin{tabular}{l}
{\em $L_K(E)$ is graded locally noetherian if and only if $L_K(E)$ is locally noetherian}\hskip.4cm while \\
{\em $L_K(E)$ can be graded locally artinian without being locally artinian.}
\end{tabular}
\end{center}
 
The paper is organized as follows. Section \ref{section_preliminaries} contains some preliminaries. Section \ref{section_direct_limits} contains Proposition \ref{ultramatricial_rep} with the structural characterization of Cohn-Leavitt algebras of no-exit objects as graded and involutive algebras and its corollaries, Corollaries \ref{no-exit_corollary}, \ref{acyclic_corollary}, \ref{comet_corollary}, and \ref{projectives}, as well as a negative answer to Realization Question, Proposition \ref{realization_question}. In sections \ref{section_noetherian} and \ref{section_artinian}, we introduce the graded versions of locally noetherian and locally artinian rings, prove some of their properties, and characterize Leavitt path algebras which are graded locally noetherian and graded locally artinian as well as locally noetherian and locally artinian without any restriction on the cardinality of the graph (Theorems \ref{noetherian} and \ref{artinian}).

\section{Preliminaries}\label{section_preliminaries}

In this section, we recall some concepts and prove some preliminary results. 

\subsection{Graded rings} 
If $\Gamma$ is an abelian group, a ring $R$ is a \emph{$\Gamma$-graded ring} if $R=\bigoplus_{ \gamma \in \Gamma} R_{\gamma}$ such that each $R_{\gamma}$ is
an additive subgroup of $R$ and $R_{\gamma}  R_{\delta} \subseteq R_{\gamma + \delta}$ for all $\gamma, \delta \in \Gamma$. If it is clear from the context what the group $\Gamma$ is, such ring $R$ is said to be a graded ring. The elements of $\bigcup_{\gamma \in \Gamma} R_{\gamma}$ are the \emph{homogeneous elements} of $R.$ If $R$ is an algebra over a field $K$, then $R$ is a \emph{graded algebra} if $R$ is a graded ring and $R_{\gamma}$ is a $K$-vector subspace for any $\gamma \in \Gamma$.

A $\Gamma$-graded ring $R$ is \emph{trivially graded} if $R_0=R$ and $R_\gamma=0$ for $0\neq\gamma \in \Gamma$. Note that any ring can be trivially graded by any abelian group.

For a $\Gamma$-graded ring $R$ and $(\gamma_1,\dots,\gamma_n)$ in $\Gamma^n$, $\M_n(R)(\gamma_1,\dots,\gamma_n)$ denotes the ring of matrices $\M_n(R)$ and the $\Gamma$-grading such that the $\delta$-component consists of the matrices $(r_{ij})\in\M_n(R)$ such that $r_{ij}\in R_{\delta+\gamma_j-\gamma_i}$ for $i,j=1,\ldots, n$ (more details in~\cite[Section 1.3]{Roozbeh_graded_ring_notes}). 

A \emph{graded left $R$-module} is a left $R$-module $M$ with a direct sum decomposition $M =\bigoplus_{\gamma\in\Gamma} M_\gamma$ where $M_\gamma$
is an additive subgroup of $M$ such that $ R_\gamma M_\delta \subseteq M_{\gamma+\delta}$ for all $\gamma,\delta\in \Gamma$. In this case, for $\delta\in\Gamma,$ the
$\delta$-\emph{shifted} graded left $R$-module $M(\delta)$ is defined as $M(\delta) =\bigoplus_{\gamma \in \Gamma} M(\delta)_\gamma,$  where $M(\delta)_\gamma = M_{\delta+\gamma}.$ A left $R$-module homomorphism $f$ of graded left $R$-modules $M$ and $N$ is a {\em graded homomorphism} if $f(M_\gamma)\subseteq N_\gamma$
for any $\gamma\in \Gamma$. 

A \emph{graded left ideal} of $R$ is a left ideal $I$ such that $I =\bigoplus_{\gamma\in\Gamma} I\cap R_\gamma.$ A left ideal $I$ of $R$ is a graded left ideal if and only if $I$ is generated by homogeneous elements. Graded right modules and graded right ideals are defined similarly. A graded ideal is a graded left and a graded right ideal.  

\subsection{Graded rings with involution} 
A ring $R$ with an involution $*$ (an anti-automorphism of order two) is said to be an {\em involutive ring} or a {\em $*$-ring}.
If a $*$-ring $R$ is also a $K$-algebra for some commutative, involutive ring $K,$ then $R$ is a {\em
$*$-algebra} if $(kr)^*=k^*r^*$ for $k\in K$ and $r\in R.$

In \cite{Roozbeh_Lia}, a $\Gamma$-graded ring $R$ with involution is said to be a \emph{graded $*$-ring} if $R_\gamma ^*\subseteq R_{-\gamma}$ for every $\gamma\in \Gamma.$ A graded ring homomorphism $f$ of graded $*$-rings $R$ and $S$ is a {\em graded $*$-homomorphism} if $f(r^*)=f(r)^*$ for every $r\in R.$ 
If $R$ is a graded $*$-ring, the \emph{$*$-transpose} $(r_{ij})^*=(r_{ji}^*)$, for $(r_{ij}) \in \M_n(R)(\gamma_1,\dots,\gamma_n)$, gives the structure of a graded $*$-ring to $\M_n(R)(\gamma_1,\dots,\gamma_n).$ A \emph{graded matricial $*$-algebra over $R$} is a finite direct sum of graded matrix algebras of the form $\M_{n}(R)(\overline \gamma)$  for $\overline \gamma\in  
\Gamma^{n}$ where the involution is the $*$-transpose in each coordinate.  

Let $\kappa$ be a cardinal and $R$ a ring. We let $\M_\kappa(R)$ denote the algebra of matrices over $R$, having rows and columns indexed by $\kappa$ and with only finitely many nonzero entries. If $\Gamma$ is an abelian group, $R$ a $\Gamma$-graded $*$-ring and $\ol \gamma: \kappa\to \Gamma$ any function, we let $\M_\kappa(R)(\ol \gamma)$ denote the $\Gamma$-graded ring $\M_\kappa(R)$ with the $\delta$-component consisting of the matrices $(r_{ij}),$ $i,j\in \kappa,$ such that $r_{ij}\in R_{\delta+\ol \gamma(j)-\ol\gamma(i)}.$ The graded ring $\M_\kappa(R)(\ol \gamma)$ is a graded $*$-ring with the $*$-transpose involution.

\subsection{Leavitt path and Cohn-Leavitt algebras.}
Let $E=(E^0, E^1, \so, \ra)$  be a (directed) graph where $E^0$ is the set of vertices, $E^1$ the set of edges, and $\so_E, \ra_E: E^1
\to E^0$ are the source and the range maps. If it is clear from the context, we write $\so_E$ and $\ra_E$ shorter as $\so$ and $\ra.$

A vertex $v$ of a graph $E$ is said to be {\em regular} if $\so^{-1}(v)$ is nonempty and finite. A vertex $v$ is a {\em sink} if $\so^{-1}(v)$ is empty and a {\em bifurcation} if $\so^{-1}(v)$ has at least two elements. A graph $E$ is \emph{row-finite} if sinks are the only vertices which are not regular,
\emph{finite} if $E$ is row-finite and $E^0$ is finite (in which case $E^1$ is necessarily finite as well), and {\em countable} if both $E^0$ and $E^1$ are countable. 

A {\em path} $p$ of $E$ is a finite sequence of edges $p=e_1\ldots e_n$ such that $\ra(e_i)=\so(e_{i+1})$ for $i=1,\dots,n-1$. Such
path $p$ has length $|p|=n.$  The maps $\so$ and $\ra$ extend to paths by $\so(p)=\so(e_1)$ and $\ra(p)=\ra(e_n)$. A vertex $v$ is a \emph{trivial} path of length zero with $\so(v)=\ra(v)=v$. 
A path $p = e_1\ldots e_n$ is \emph{closed} if $\so(p)=\ra(p)$. If $p=e_1\ldots e_n$  is a closed path and $\so(e_i) \neq \so(e_j)$ for all $i \neq j$, then $p$ is a \emph{cycle}. A cycle of length one is a {\em loop}. A graph $E$ is {\em no-exit} if no vertex of any cycle is a bifurcation. 

An infinite path of a graph $E$ is a sequence of edges $e_1e_2\ldots$ such that $\ra(e_i)=\so(e_{i+1})$ for $i=1,2,\ldots$. An infinite path is an \emph{infinite sink} if none of its vertices are bifurcations or in a cycle. An infinite path \emph{ends in a sink} if there is a positive integer $n$ such that the subpath $e_ne_{n+1}\hdots$ is an infinite sink. An infinite path \emph{ends in a cycle} if there is a positive integer $n$ such that the subpath $e_ne_{n+1}\hdots$ is equal to the path $cc\hdots$ for some cycle $c.$ 

Extend a graph $E$ by the new edges $\{e^*\ |\ e\in E^1\}$ such that $\so(e^*)=\ra(e)$  and $\ra(e^*)=\so(e)$ for all edges $e.$ Extend the map $^*$ to paths by defining $v^*=v$ for vertices $v$ and $(p)^*=e_n^*\ldots e_1^*$ for paths $p=e_1\ldots e_n, e_i\in E^1, i=1,\ldots,n$ and extend $\so$ and $\ra$ by $\so(p^*)=\ra(p)$ and $\ra(p^*)=\so(p)$.  

For a graph $E$ and a field $K$, the \emph{Leavitt path algebra} $L_K(E)$ of $E$ over $K$ is the free $K$-algebra generated by the set  $E^0\cup E^1\cup\{e^*\ |\ e\in E^1\}$ such that for all vertices $v,w$ and edges $e,f,$
\begin{itemize}
\item[(V)]  $vw =0$ if $v\neq w$ and $vv=v,$

\item[(E1)]  $\so(e)e=e\ra(e)=e,$

\item[(E2)] $\ra(e)e^*=e^*\so(e)=e^*,$

\item[(CK1)] $e^*f=0$ if $e\neq f$ and $e^*e=\ra(e),$

\item[(CK2)] $v=\sum_{e\in \so^{-1}(v)} ee^*$ for each regular vertex $v.$
\end{itemize}

The \emph{Cohn path algebra $C_K(E)$ of $E$ over $K$} is the free $K$-algebra generated by $E^0\cup E^1\cup \{e^*\ |\ e\in E^1\}$ such that (V), (E1), (E2), and (CK1) axioms hold for all vertices $v,w$ and edges $e,f$.

The Cohn-Leavitt algebras are obtained from Cohn path algebras by requiring the (CK2) axiom to hold just for a portion of regular vertices, not necessarily all of them. Thus, the Cohn-Leavitt algebras of $E$ over $K$ can be considered to be ``$C_K(E),$ $L_K(E)$ and everything in between''. More precisely, if $S$ is a subset of regular vertices, the {\em Cohn-Leavitt algebra $CL_K(E,S)$ of $E$ and $S$ over $K$}  is the free $K$-algebra generated by the sets $E^0\cup E^1\cup \{e^*\ |\ e\in E^1\}$ subject to relations (V), (E1), (E2), (CK1) for all vertices $v,w$ and edges $e,f$ and 
\begin{itemize}
\item[(SCK2)] $v=\sum_{e\in \so^{-1}(v)} ee^*,$ for every vertex $v\in S$.
\end{itemize}

The Cohn-Leavitt algebra $CL_K(E,\varnothing)$ is a Cohn path algebra and we write $CL_K(E, \varnothing)$ as $C_K(E).$  The algebra $CL_K(E,R(E))$ is a Leavitt path algebra and we write $CL_K(E, R(E))$ as $L_K(E).$   

The axioms (V), (E1), (E2), (CK1)  imply that every element of $CL_K(E,S)$ can be represented as a sum of the form $\sum_{i=1}^n a_ip_iq_i^*$ for some $n$, paths $p_i$ and $q_i$, and elements $a_i\in K,$ for $i=1,\ldots,n.$ By the same axioms, $CL_K(E, S)$ is a unital ring if and only if $E^0$ is finite (in this case the identity is the sum of elements of $E^0$). If $E^0$ is not finite, the finite sums of distinct vertices are local units of $CL_K(E).$ 

If $K$ is a field with involution $*$ (and there is always at least one such involution, the identity), $CL_K(E,S)$ becomes a $*$-algebra by 
$\left(\sum_{i=1}^n a_ip_iq_i^*\right)^* =\sum_{i=1}^n a_i^*q_ip_i^*$ for $a_i\in K$ and paths $p_i, q_i, i=1,\ldots, n.$
The algebra  $CL_K(E,S)$ is also naturally graded by $\Zset$ so that the $n$-component is  
\[CL_K(E, S)_n=  \Big \{ \sum_i a_i p_i q_i^*\mid p_i, q_i \textrm{ are paths}, a_i \in K, \textrm{ and } |p_i|-|q_i|=n \textrm{ for all } i \Big\}.\]
Since $CL_K(E, S)_n^*=CL_K(E, S)_{-n}$ for every $n\in\Zset,$ $CL_K(E, S)$ is a graded $*$-algebra. The field $K$ is assumed to be trivially graded by $\Zset.$ We keep this assumption throughout the paper. 

In \cite{Ara_Goodearl}, Ara and Goodearl introduced the Cohn-Leavitt algebras for a class of separated graphs, which is a strictly larger class of graphs than the class of directed graphs. A directed graph is consider to be trivially separated. The $C^*$-analog of Cohn-Leavitt algebras over trivially separated graphs, 
the {\em relative graph $C^*$-algebra} $C^*(E,S)$ of a graph $E$ and $S\subseteq R(E)$, was introduced in \cite{Muhly_Tomforde}. 

Based on the definitions alone, one could suspect that the class of Cohn-Leavitt algebras of trivially separated graphs is strictly larger than the class of Leavitt path algebras. However, that is not the case: there is a canonical $*$-isomorphism of $CL_K(E, S)$ and $L_K(E_S)$ for a suitable graph $E_S$ defined via $E$ and $S$ (see \cite[Theorem 3.7]{Muhly_Tomforde} for $E$ countable and \cite[Lemma 4.8]{Lia_Traces} for any $E$). 

Although the classes of Leavitt and Cohn-Leavitt algebras are the same, the consideration of Cohn-Leavitt algebras has (at least) one significant advantage over the consideration of Leavitt path algebras alone: every Leavitt path algebra is a direct limit of Cohn-Leavitt algebras of certain finite subgraphs {\em with injective connecting maps} by \cite[Proposition 3.6]{Ara_Goodearl}. This fact enables one to transfer the consideration of certain properties of algebra to subalgebras corresponding to {\em finite} graphs. Using Leavitt path algebras alone, such direct limit representation is not always possible.   

We start by showing that some results of \cite{Ara_Goodearl} and \cite{Lia_Traces}, formulated for non-graded algebras, continue to hold in the category of graded involutive algebras as well. 

Let $E$ be a graph with $S\subseteq R(E)$ and $F$ a subgraph of $E$ with $T\subseteq R(F).$ We say that $(F, T)$ is a {\em complete subobject} of
$(E,S)$ if $T\subseteq S$ and the following holds.
\begin{enumerate}
\item[(C)] If $v\in S\cap F^0$ is such that  $\so^{-1}_E(v)\cap F^1\neq \varnothing,$ then $\so^{-1}_F(v)=\so^{-1}_E(v)$ and $v\in T.$
\end{enumerate}
Note that the conditions $T\subseteq S$ and (C) imply that $T=S\cap\{v\in F^0\, |\, \so^{-1}_E(v)\cap F^1\neq \varnothing\}.$

If $T=R(F)$ and $S=R(E),$ this agrees with the definitions of a complete subgraph for row-finite graphs from \cite[Section 3]{Ara_Moreno_Pardo} and for countable graphs from \cite[Definition 9.7]{Abrams_Tomforde}. In this case, the inclusion of $F$ into $E$ induces a {\em graded} homomorphism as pointed out in \cite[\S 2]{Roozbeh_Israeli}.

In \cite[Proposition 3.5]{Ara_Goodearl} it is shown that for a finite subgraph $G$ of a graph $E$ and any $S\subseteq R(E),$ there is a complete subobject $(F,T)$ of $(E,S)$ such that $F$ is finite and $G$ is a subgraph of $F.$ In the case when $S=R(E),$ $E$ has an infinite emitter $v$ and $G$ is a subgraph consisting of $v$ with finitely many edges $v$ emits together their ranges, such complete subobject $(F,T)$ is with $v\in R(F)$ but $v\notin T.$ Thus $CL_K(F,T)$ is a $K$-subalgebra of $L_K(E)$ while $L_K(F)$ is not. Cases like this one highlight advantages of working with Cohn-Leavitt algebras instead of Leavitt path algebras alone. 

Consider the category $\G$ whose objects are pairs $(E, S),$ where $E$ is a graph and $S\subseteq R(E)$ and whose morphisms $(E, S)\to (F, T)$ are graph morphisms 
$f:E\to F$ such that (1) $f$ is a graph morphisms which is injective on the set of vertices, (2) $v$ emits edges if and only if $f(v)$ emits edges for every $v\in E^0$ and $f$ is injective on $\so_E^{-1}(v)$ for every $v\in E^0$ which emits edges, and (3) $f$ maps $S$ into $T$ and, if $v\in S,$ then $f$ maps $\so_E^{-1}(v)$ bijectively onto $\so_F^{-1}(f(v)).$ 

By \cite[Proposition 3.3]{Ara_Goodearl}, the category $\G$ admits arbitrary direct limits. The following proposition, which is the graded version of \cite[Proposition 3.6]{Ara_Goodearl} for trivially separated graphs, holds since all the maps defined in the proof of \cite[Proposition 3.6]{Ara_Goodearl} preserve the degrees of elements and, hence, are graded homomorphisms. 

\begin{proposition} The assignment $(E, S)\to CL_K(E, S)$ extends to a continuous functor $CL_K$ from the category $\G$ to the category of graded $*$-algebras over $K$. Moreover, every algebra $CL_K(E, S)$ is a direct limit, with injective connecting maps, of the algebras $CL_K(F, T)$ where the direct limit is taken over all objects $(F, T)$ of the directed system of finite complete subobjects of $(E, S).$ 
\label{dir_lim_of_fin_subobject}
\end{proposition}

Next, we briefly review the construction of the {\em relative graph} $E_S$ of $E$ with respect to $S$ from \cite[Theorem
3.7]{Muhly_Tomforde}. The graph $E_S$ is defined by 
\begin{align*}
E^0_S & =E^0\cup\{v' \,|\, v\in R(E)- S\}\\
E^1_S & =E^1\cup \{e' \,|\, e\in E^1\mbox{ with }\ra(e)\in R(E)- S\}
\end{align*}
and by letting the maps $\so$ and $\ra$ in $E_S$ be the same as in $E$ on $E^1$ and such that $\so(e')=\so(e)$ and
$\ra(e')=\ra(e)'$ for any added edge $e'.$ 

Define a map $\phi_{E,S}$ on $E_S^0$ by $\phi_{E,S}(v)=v$ if $v\notin R(E)- S$,
$\phi_{E,S}(v)=\sum_{e\in \so^{-1}(v)} ee^*$ and
$\phi_{E,S}(v')=v-\sum_{e\in \so^{-1}(v)} ee^*$ if $v\in R(E)- S.$
Define $\phi_{E,S}$ on $E_S^1$ by $\phi_{E,S}(e)=e\phi_{E,S}(\ra(e))$ for $e\in E^1$ and
$\phi_{E,S}(e')=e\phi_{E,S}(\ra(e)')$ for $e\in E^1$ such that $\ra(e)\in R(E)- S.$ If  $f\in E_S^1,$ let 
$\phi_{E,S}(f^*)=\phi_{E,S}(f)^*.$  

It can be directly checked that the map $\phi_{E,S}$ is such that the images $\phi_{E,S}(w)$, $\phi_{E,S}(f),$ and $\phi_{E,S}(f^*)$ for $w\in E^0_S$ and $f\in E^1_S$ satisfy (V), (E1), (E2), (CK1), and (CK2). Thus, $\phi_{E,S}$ uniquely extends to a $K$-algebra $*$-homomorphism of $L_K(E_S)$ to $CL_K(E,S)$ by the universal property of Leavitt path algebras (see \cite[Lemma 4.7]{Lia_Traces} for the involutive version of this property). Moreover, since the map $\phi_{E,S}$ respects the grading on vertices, edges and ghost edges, its extension is a graded homomorphism which is injective by the Graded Uniqueness Theorem (more details in \cite[Lemma 4.8]{Lia_Traces}). This map is onto by \cite[ Theorem 3.7]{Muhly_Tomforde} or \cite[Lemma 4.8]{Lia_Traces}.

Moreover, the map $\phi_{E,S}$ is canonical. Indeed, if $f$ is a graded $*$-homomorphism of algebras $f: CL_K(E,S)\to CL_K(F, T),$ 
one can check that the map $\overline f$ defined on $E_S^0$ and $E_S^1$ by
\[\begin{array}{lll}
v\mapsto &\phi^{-1}_{F,T}f(\phi_{E,S}(v)) & \mbox{ for }v\in E^0\\
v'\mapsto &\phi^{-1}_{F,T}(f(v)-f(\phi_{E,S}(v)))  &\mbox{ for }v\in R(E)-S\\
e \mapsto & \phi^{-1}_{F,T}(f(e)f(\phi_{E,S}(\ra(e))))  & \mbox{ for }e\in E^1\\
e'\mapsto &\phi^{-1}_{F,T}(f(e)f(\phi_{E,S}(\ra(e)'))) & \mbox{ for }e\in E^1\mbox{ with }\ra(e)\in R(E)-S
\end{array}
\]
extends to a graded $*$-homomorphism $\ol f: L_K(E_S)\to L_K(F_T).$ Thus, we have the following lemma. 

\begin{lemma}\label{star_iso}
The map $\phi_{E,S}: L_K(E_S)\cong_{\gr} CL_K(E,S)$ is a canonical graded $*$-isomorphism. 
\end{lemma}

\section{Involutive and graded structure of Leavitt path algebras of no-exit graphs}\label{section_direct_limits}

In this section, we characterize Cohn-Leavitt algebras of no-exit graphs as {\em graded} and {\em involutive} algebras without any restriction on the cardinality of the graph or the infinite paths (Proposition \ref{ultramatricial_rep}). Our result is an involutive and graded version of \cite[Theorem 3.2, equivalences (i) and (ii)]{Gonzalo_Mercedes_Brox} and a generalization of \cite[Theorem 6.7, equivalences (2) and (3)]{Roozbeh_Ranga}. Corollaries of this result include graph-theoretic conditions characterizing when a Leavitt path algebra is a directed union of (graded) matricial algebras over the underlying field and over the algebra of Laurent polynomials and when the monoid of isomorphism classes of finitely generated projective modules is atomic and cancellative.

\subsection{Benefits of using Cohn-Leavitt algebras}\label{subsection_necessity}
By \cite[Proposition 5.1]{Roozbeh_Lia}, if $E$ is a row-finite, no-exit graph in which each infinite path ends in a sink or a cycle and $K$ is any field, $L_K(E)$ is graded $*$-isomorphic to the algebra 
$$\bigoplus_{i\in I} \M_{\kappa_i} (K)(\ol\alpha^i) \oplus \bigoplus_{j \in J} \M_{\mu_j} (K[x^{n_j},x^{-n_j}])(\ol\gamma^j)$$
where $I$ indexes the set of sinks, $J$ indexes the set of cycles, $\kappa_i$ is the number of paths ending in a sink indexed by $i\in I,$ $\mu_j$ the number of paths ending in a fixed but arbitrary vertex of the cycle indexed by $j\in J,$ and $n_j$ corresponds to the length of the $j$-th cycle, $j\in J.$ A bit of care is needed when dealing with infinite sinks and \cite[Section 5.2]{Roozbeh_Lia} contains more details as well as examples. 

The $k$-th entry of the shift $\ol\alpha^i\in \Zset^{\kappa_i},$ for $k\in \kappa_i,$ $i\in I,$ corresponds to the length of the $k$-th path ending in the $i$-th sink.  The $k$-th entry of the shift $\ol\gamma^j\in \Zset^{\mu_j},$ for $k\in \mu_j,$ $j\in J,$ corresponds to the length of the $k$-th path ending in a fixed (but arbitrary) vertex of the $j$-th cycle. A bit more care is needed when handling infinite sinks and \cite[Section 5.2]{Roozbeh_Lia} contains more details. 

Before proving any results, we consider the following example. It illustrates that the use of Cohn-Leavitt algebras is necessary when representing a Leavitt path algebra $L_K(E)$ as a direct limit of subalgebras of its finite subgraphs in the case when $E$ fails to be row-finite. 

\begin{example}
Let $E$ be the graph 
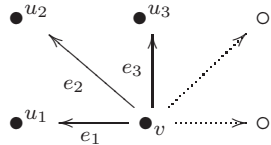
\[\xymatrix{
\bullet^{u_2} & \bullet^{u_3} & \circ\\
\bullet^{u_1}  & \bullet_v \ar[l]^{e_1} \ar[ul]^{e_2} \ar[u]^{e_3} \ar@{.>}[ur] \ar@{.>}[r]  &  \circ   \\
}\]
with one source $v$ emitting infinitely many edges $e_1,e_2,\ldots$ with $e_n$ ending in a sink $u_n,$ $n=1,2,\ldots$ as in the figure above. Note that $E$ has no regular vertices so every subgraph $F$ is such that $(F, \varnothing)$ is a complete subobject of $(E, \varnothing).$ Let $F_n, n=1,2,\ldots$ denote the family of subgraphs 
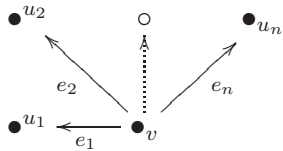
\[\xymatrix{
\bullet^{u_2}& \circ& \bullet_{u_n}\\
\bullet^{u_1}& \bullet_v \ar[l]^{e_1} \ar[ul]^{e_2} \ar@{.>}[u]\ar[ur]_{e_n}  &  
}\]
generated by the edges $e_1, e_2,\ldots, e_n.$ Then $(F_n, \varnothing)$ is a complete subobject of both  $(F_{n+1}, \varnothing)$ and $(E, \varnothing).$ The algebra  $C_K(F_n)$ is graded $*$-isomorphic to $\M_2(K)(0,1)^n\oplus K.$ Indeed, if $e_{jl}^i, j,l=1,2,$  denote the standard matrix units in the $i$-th copy of the matricial algebra  $\M_2(K)(0,1)^n\oplus K,$ for $i=1,\ldots, n,$ then the map below induces a graded $*$-isomorphism.  
\[u_i\mapsto (e_{11}^i, 0), \;\; e_i\mapsto (e_{21}^i, 0),\; i=1,\ldots, n,\;\;\;\;\;\; v\mapsto (\sum_{i=1}^n e_{22}^i, 1)\]
Thus, the element $(0,1)$ corresponds to $v-\sum_{i=1}^n e_ie_i^*.$

Let $\phi_{n(n+1)}: C_K(F_n)\to C_K(F_{n+1})$ denote the graded $*$-monomorphism induced by the inclusion $(F_n, \varnothing)\to (F_{n+1}, \varnothing),$ and $R$ denote the direct limit of the system $\{(C_K(F_n), \phi_{n(n+1)})\mid n=1,2\ldots\}$ with the translational maps $\phi_n: C_K(F_n)\to R.$
Also, let $\psi_n: C_K(F_n)\to C_K(E)=L_K(E)$ denote the graded $*$-monomorphism induced by the inclusion   $(F_n, \varnothing)\to (E, \varnothing).$ Since 
$\psi_n=\phi_{n(n+1)}\psi_{n+1},$ there is a graded $*$-monomorphism $\psi: R\to L_K(E)$ such that $\psi_n=\psi\phi_n$ by the universal property of the direct limit. 

The inverse $\phi$ of the map $\psi$ can be obtained by noting that the correspondence 
\[u_n\mapsto \phi_n(u_n), v\mapsto v=\phi_n(v), e_n\mapsto \phi_n(e_n),\text{ for }n=1,2,\ldots\]
defines a universal Cuntz-Krieger $(E, \varnothing)$-family and, hence, a graded $*$-homomorphism $\phi: L_K(E)\to R.$ 
The relation $\phi\psi_n=\phi_n$ holds by definition on $F_n^0$ and $F_n^1$ and, consequently, $\phi\psi_n=\phi_n$ holds on $C_K(F_n)$ for every $n.$ This implies that $\phi\psi$ is the identity on $R.$ For the converse, note that if $x$ is any of $u_i, v,$ or $e_i$ for $i=1,2,\ldots, n,$ for any $n$, then $\psi(\phi(x))=\psi(\phi_n(x))=\psi_n(x)=x.$ This implies that $\psi\phi$ is the identity. 

Using a matricial representation of $C_K(F_n),$ the algebra $R$ can be identified with the graded $*$-algebra $\varinjlim_n \left(\M_2(K)(0,1)^n\oplus K\right)$ 
with the connecting maps given by $(x,0)\mapsto (x, 0)$ for $x\in \M_2(K)(0,1)^n$ and $(0,1)\mapsto (e^{n+1}_{22}, 1)$ and we have the following commutative diagram. 
{\small
\[
\xymatrix{
{\ldots} \ar[r] &  {C_K(F_n)}\ar[r]\ar[d]^{\cong_{\gr}}  & {C_K(F_{n+1}) }\ar[d]^{\cong_{\gr}}\ar[r] & {\ldots}\ar[r] & {\;L_K(E)\;} \ar[d]^{\cong_{\gr}} \\ 
{\ldots} \ar[r] & {\M_2(K)(0,1)^n\oplus K}\ar[r]        & {\M_2(K)(0,1)^{n+1}\oplus K}\ar[r]\ar[r] & {\ldots}\ar[r] & {\varinjlim_n\left(\M_2(K)(0,1)^n\oplus K\right)} }
\]}
Note that the algebra $L_K(E)$ cannot be represented as a directed union of Leavitt path algebras of any of its finite, complete subgraphs. 
Note also that the graph $E$ is acyclic. Thus, \cite[Theorem 6.7]{Roozbeh_Ranga} is not applicable to this graph but, as we shall see, Corollary \ref{acyclic_corollary} is. 
\label{example1}
\end{example}

\subsection{The structure of Cohn-Leavitt algebras of no-exit objects} 

If $E$ is a no-exit graph and $S\subseteq R(E),$ the relative graph $E_S$ is also a no-exit graph provided that the following condition holds. 
\begin{itemize}
\item[$(*)$] Vertices of every cycle of $E$ are in $S$.
\end{itemize}
We say that the object $(E,S)$ is a {\em no-exit object} of the category $\G$ if $E$ is a no-exit graph and condition $(*)$ holds. We consider finite, no-exit objects of $\G$ in the next lemma.   

\begin{lemma} If $(E,S)$ is a finite, no-exit object of $\G$, then 
$CL_K(E,S)$ is graded $*$-isomorphic to the algebra 
$$\bigoplus_{i=1}^k \M_{k_i} (K)(\ol\alpha^i) \oplus \bigoplus_{j=1}^m \M_{m_j} (K[x^{n_j},x^{-n_j}])(\ol\gamma^j)$$
where $k$ is the number of sinks and vertices in $R(E)-S$, $k_i$ is the number of paths ending in the sink or the vertex in $R(E)-S$ indexed by $i$ for $i=1,\ldots, k,$ $\ol\alpha^i(j)$ is the length of the $j$-th path ending in the $i$-th vertex for $j=1,\ldots, k_i$ and $i=1,\ldots, k,$ $m$ is the number of cycles, $m_j$ the number of paths ending in a fixed but arbitrary vertex of the cycle indexed by $j,$ $n_j$ is the length of the $j$-th cycle for $j=1,\ldots, m,$ and $\ol\gamma^j(l)$ is the length of the $l$-th path ending in the fixed vertex of the $j$-th cycle for $l=1,\ldots, m_j$ and $j=1,\ldots, m.$
\label{finite_no-exit} 
\end{lemma}
\begin{proof} 
Since $(E,S)$ is a finite, no-exit object, the relative graph $E_S$ is a finite, no-exit graph. Several different papers relate the Leavitt path algebra of a finite, no-exit graph with a direct sum matricial algebras over $K$ and over $K[x^n,x^{-n}]$ for positive integers $n.$ In particular, there is an algebra isomorphism by \cite[Theorem 3.7]{AAPM}, a $*$-algebra isomorphism by \cite[Corollary 32]{Zak_Lia}, a graded algebra isomorphism by \cite[Theorem 6.7]{Roozbeh_Ranga}, and a {\em graded $*$-algebra isomorphism} by \cite[Proposition 5.1]{Roozbeh_Lia}. Thus, $L_K(E_S)$ is graded $*$-isomorphic to an algebra as in the statement of the lemma. Let $\phi_{E_S}$ denote this graded $*$-isomorphism. 

By construction of the relative graph and the assumption that $(*)$ holds, the cycles and paths which end in them correspond to the same elements in $(E,S)$ and in $E_S.$ The number of sinks of $E_S$ is equal to the number of sinks in $E$ plus the number of vertices in $R(E)-S.$ By construction, the sinks and paths which end in them correspond to the same elements in $(E,S)$ and in $E_S$ and the number of paths in $E$ which end in $v\in R(E)-S$ correspond exactly to the number of paths in $E_S$ which end in the sink $v'.$ The length of a path ending in $v\in R(E)-S$ and the length of the corresponding path ending in $v'\in E_S^0$ are the same.  
Let $\phi_{E,S}$ be the graded $*$-isomorphism $L_K(E_S)\to CL_K(E,S)$ from Lemma \ref{star_iso}. Then $\phi_{E_S}\phi_{E,S}^{-1}$ is a graded $*$-isomorphism of $CL_K(E,S)$ and the direct sum of matricial algebras as in the statement of the lemma. 
\end{proof}

\begin{proposition}\label{ultramatricial_rep}
Let $E$ be a graph and $S\subseteq R(E)$ such that $(E,S)$ is a no-exit object of the category $\G$.  
Then $CL_K(E,S)$ is graded $*$-isomorphic to a direct limit 
$$\varinjlim_{i\in I} \left(\bigoplus_{l=1}^{k_i} \M_{k_{il}} (K)(\ol\alpha^i_l) \oplus \bigoplus_{j=1}^{m_i} \M_{m_{ij}} (K[x^{n_{ij}},x^{-n_{ij}}])(\ol\gamma^i_j)\right)$$ with injective connecting maps where $I$ is a directed set, $k_i, m_i$ are nonnegative integers and $k_{ij},$ $m_{il},$ and $n_{ij}$ positive integers, $\ol \alpha^i_l\in \Zset^{k_{il}}$ and $\ol\gamma^i_j\in \Zset^{m_{ij}},$ for $i\in I,$ $l=1, \ldots k_i,$ and $j=1,\ldots, m_i$. 

In particular, if $E$ is a no-exit graph and $S=R(E),$ the condition $(*)$ holds, so the conclusion holds for $L_K(E).$ 
\end{proposition}
\begin{proof}
The object $(E,S)$ can be represented as a direct limit of a family of finite, complete subobjects $(F_i, S_i)_{i\in I}$ by \cite[Proposition 3.5]{Ara_Goodearl}. Each subobject $(F_i, S_i)$ is finite and no-exit since $(E,S)$ is no-exit and the vertices of all cycles of $F_i$ are in $S_i$ by completeness. By Proposition \ref{dir_lim_of_fin_subobject}, $CL_K(E,S)$ is graded $*$-isomorphic to the direct limit of the algebras $CL_K(F_i, S_i)$ over $I$ with injective connecting maps. 
By Lemma \ref{finite_no-exit}, $CL_K(E, S)$ is graded $*$-isomorphic to an algebra of the required form. 
\end{proof}

Proposition \ref{ultramatricial_rep} is related to \cite[Theorem 3.2]{Gonzalo_Mercedes_Brox} except that we use Cohn-Leavitt algebras of the subgraphs of the graph and \cite[Theorem 3.2]{Gonzalo_Mercedes_Brox} uses the Leavitt path algebras corresponding to the duals of the subgraphs, not the subgraphs themselves. Moreover, we consider both the graded and the involutive structure of the algebras. 

If $(E,S)$ is a no-exit object, we shall say that the graded $*$-algebra from Proposition \ref{ultramatricial_rep} is a {\em graded locally matricial representation} of the algebra $CL_K(E,S).$ 
The next corollary shows that no-exit graphs are the {\em only} graphs with Leavitt path algebras graded $*$-isomorphic to a direct limit 
as in Proposition \ref{ultramatricial_rep}.

\begin{corollary} Let $E$ be an arbitrary graph, $K$ a field, $R^{\gr}$ a graded $*$-algebra of the form 
$$\varinjlim_{i\in I} \left(\bigoplus_{l=1}^{k_i} \M_{k_{il}} (K)(\ol\alpha^i_l) \oplus \bigoplus_{j=1}^{m_i} \M_{m_{ij}} (K[x^{n_{ij}},x^{-n_{ij}}])(\ol\gamma^i_j)\right),$$ and $R$ a $*$-algebra of the form 

$$\varinjlim_{i\in I} \left(\bigoplus_{l=1}^{k_i} \M_{k_{il}} (K) \oplus \bigoplus_{j=1}^{m_i} \M_{m_{ij}} (K[x,x^{-1}])\right),$$
where $I,$ $k_i, m_i,$ $k_{il},$ $m_{ij},$ $n_{ij}$ $\ol \alpha^i_l,$ and $\ol\gamma^i_j$ are as in Proposition \ref{ultramatricial_rep} for $i\in I,$ $l=1, \ldots k_i,$ and $j=1,\ldots, m_i$. 
The following conditions are equivalent. 
\begin{enumerate}[\upshape(a)]
\item $L_K(E)$ is graded $*$-isomorphic to $R^{\gr}$.
\item $L_K(E)$ is graded isomorphic to $R^{\gr}.$ 
\item $L_K(E)$ is $*$-isomorphic to $R.$ 
\item $L_K(E)$ is isomorphic to $R.$
\item $E$ is no-exit.
\end{enumerate} 
Let $S^{\gr}$ denote a graded $*$-algebra of the form 
$$\bigoplus_{i\in I} \M_{\kappa_i} (K)(\ol\alpha^i) \oplus \bigoplus_{j \in J} \M_{\mu_j} (K[x^{n_j},x^{-n_j}])(\ol\gamma^j)$$
and $S$ a $*$-algebra of the form 
$$\bigoplus_{i\in I} \M_{\kappa_i} (K) \oplus \bigoplus_{j \in J} \M_{\mu_j} (K[x,x^{-1}])$$
where $I, J$ are sets, $\kappa_i, i\in I$ and $\mu_j, j\in J$ cardinals, $n_j$ positive integers for $j\in J,$  $\ol \alpha^i\in \Zset^{\kappa_i}$ for $i\in I,$ and $\ol\gamma^j\in \Zset^{\mu_j}$ for $j\in J.$  

Let (1) to (4) denote the conditions (a) to (d) with $R^{\gr}$ replaced by $S^{\gr}$ and $R$ replaced by $S$ and (5) denote the condition below.
\begin{enumerate}
\item[(5)] $E$ is a row-finite, no-exit graph such that every infinite path ends in a sink or a cycle.
\end{enumerate}
Then conditions (1) to (5) are equivalent. 
\label{no-exit_corollary}
\end{corollary}
\begin{proof}
The implications (a) $\Rightarrow$ (b) $\Rightarrow$ (d) and (a) $\Rightarrow$ (c) $\Rightarrow$ (d) are direct and (e) $\Rightarrow$ (a) follows from Proposition \ref{ultramatricial_rep}. So, it remains to show (d) $\Rightarrow$ (e). Assuming (d), it is direct to check that $L_K(E)$ is directly finite (in the sense that for any $x,y$ with a local unit $u$ such that $xu=ux=x$ and $yu=uy=y,$ if $xy=u$ then $yx=u$) since the connecting maps are injective. By \cite[Theorem 4.12]{Lia_Traces}, $E$ is a no-exit graph.

The implications (1) $\Rightarrow$ (2) $\Rightarrow$ (4) and (1) $\Rightarrow$ (3) $\Rightarrow$ (4) are direct. Assuming (4), it is direct to check that $L_K(E)$ is a locally Baer ring ($pL_K(E)p$ is Baer for every idempotent $p$). By \cite[Theorem 15]{Roozbeh_Lia_Baer}, this implies condition (5). 
The implication (5) $\Rightarrow$ (1) holds by \cite[Proposition 5.1]{Roozbeh_Lia}. 
\end{proof}

Proposition \ref{ultramatricial_rep} also implies that the acyclic graphs are the only graphs with Leavitt path algebras graded $*$-isomorphic to a directed union of graded matricial $*$-algebras over $K.$

\begin{corollary} Let $E$ be an arbitrary graph, $K$ a field, $R^{\gr}$ a graded $*$-algebra of the form 
$$\varinjlim_{i\in I} \bigoplus_{l=1}^{k_i} \M_{k_{il}} (K)(\ol\alpha^i_l),$$ and $R$ a $*$-algebra of the form 
$$\varinjlim_{i\in I} \bigoplus_{l=1}^{k_i} \M_{k_{il}} (K),$$
where $I,$ $k_i,$ $k_{il},$ and $\ol \alpha^i_l$ are as in Proposition \ref{ultramatricial_rep} for $i\in I$ and $l=1, \ldots k_i,$ 
The following conditions are equivalent. 
\begin{enumerate}[\upshape(a')]
\item $L_K(E)$ is graded $*$-isomorphic to $R^{\gr}.$
\item $L_K(E)$ is graded isomorphic to $R^{\gr}.$
\item $L_K(E)$ is $*$-isomorphic to $R.$
\item $L_K(E)$ is isomorphic to $R.$
\item $E$ is acyclic.
\end{enumerate} 
Let $S^{\gr}$ denote a graded $*$-algebra of the form 
$\bigoplus_{i\in I} \M_{\kappa_i} (K)(\ol\alpha^i)$
and $S$ a $*$-algebra of the form 
$\bigoplus_{i\in I} \M_{\kappa_i} (K)$
where $I$ is a set, $\kappa_i$ cardinals, and $\ol \alpha^i\in \Zset^{\kappa_i}$ for $i\in I.$

Let (1') to (4') denote the conditions (a') to (d') with $R^{\gr}$ replaced by $S^{\gr}$ and $R$ replaced by $S$ and let (5') denote the condition below.
\begin{enumerate}
\item[(5')] $E$ is a row-finite, acyclic graph such that every infinite path ends in a sink.
\end{enumerate}
Then conditions (1') to (5') are equivalent. 
\label{acyclic_corollary}
\end{corollary}
\begin{proof}
The implications (a') $\Rightarrow$ (b') $\Rightarrow$ (d') and (a') $\Rightarrow$ (c') $\Rightarrow$ (d') are direct.  Since a direct limit of matricial algebras over $K$ is a regular ring, the implication (d') $\Rightarrow$ (e') follows from \cite[Theorem 1]{Gene_Ranga} stating that $L_K(E)$ is regular if and only if $E$ is acyclic. The implication (e') $\Rightarrow$ (a') follows from the fact that the absence of cycles in a no-exit graph implies the absence of the Laurent polynomial algebras in the graded locally matricial representation from Proposition \ref{ultramatricial_rep}.

The implications (1') $\Rightarrow$ (2') $\Rightarrow$ (4') and (1') $\Rightarrow$ (3') $\Rightarrow$ (4') are direct. Assuming (4'), it follows that $L_K(E)$ is a locally Baer and regular ring. By \cite[Theorem 15]{Roozbeh_Lia_Baer}, $E$ is a row-finite, no-exit graph and every infinite path ends in a sink or a cycle. The regularity of $L_K(E)$ implies that $E$ is acyclic by \cite[Theorem 1]{Gene_Ranga} so (5') holds. The implication (5') $\Rightarrow$ (1') holds by Corollary \ref{no-exit_corollary} since the absence of cycles in a graph as in (5') implies the absence of the Laurent polynomial terms in the representation from Corollary \ref{no-exit_corollary}. 
\end{proof}

The graph in Example \ref{example1} is acyclic but it does not satisfy condition (5') of Corollary \ref{acyclic_corollary}. By Corollary \ref{acyclic_corollary}, Leavitt path algebra of this graph is graded $*$-isomorphic to a direct limit of matricial algebras over $K$ but not to a direct sum of algebras of the form $\M_\kappa(K)(\ol \gamma)$ where $\kappa$ is a cardinal and $\ol\gamma\in \Zset^\kappa$. The same conclusion can be drown also for the graph in the following example.  

Let $E$ be the graph in the figure below.
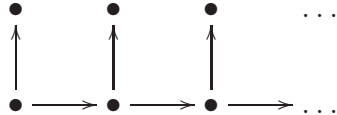
\[\xymatrix{
\bullet & \bullet& \bullet  & \dots & \\
\ar[u]\bullet \ar[r] & \ar[u]\bullet \ar[r]& \ar[r]
\ar[u]\bullet& \ar@{.}\dots
}\] The Leavitt path algebra of this graph is graded $*$-isomorphic to $$\varinjlim_n \left(\bigoplus_{i=1}^{n} \M_{i+1}(K)(0,1,2,\ldots, i)\oplus \M_n(K)(0,1,\ldots, n-1)\right).$$
By Corollary \ref{acyclic_corollary}, $L_K(E)$ is not graded $*$-isomorphic to an algebra as $S^{\gr}$ in Corollary \ref{acyclic_corollary}.

Considering Corollaries \ref{no-exit_corollary} and \ref{acyclic_corollary}, one could suspect that the Leavitt path algebra of a no-exit graph without sinks
necessarily lacks matricial algebras over $K$ in its graded locally matricial representation. 
Yet, this is not the case. Consider the graph below 
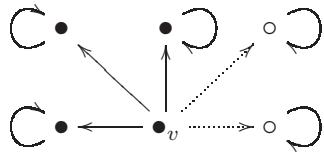
\[\xymatrix{
\bullet\ar@(ld, lu) & \bullet\ar@(ru, rd) & \circ\ar@(ru, rd)\\
\bullet\ar@(ld, lu)  & \bullet_v \ar[l] \ar[ul] \ar[u] \ar@{.>}[ur] \ar@{.>}[r]  &  \circ \ar@(ru, rd)  \\
}\]
with an infinite emitter $v$ emitting countably many edges to vertices each of which emits a single loop. This graph is no-exit and without sinks and has a graded locally matricial representation 
\[\varinjlim_n \left(\M_2(K[x,x^{-1}])(0,1)^n\oplus K\right)\] 
which can be seen similarly as in Example \ref{example1}.
However, the following corollary of Proposition \ref{ultramatricial_rep} holds.  

\begin{corollary} Let $E$ be an arbitrary graph, $K$ a field, $S^{\gr}$ a graded $*$-algebra of the form 
$$\bigoplus_{i\in I} \M_{\mu_i} (K[x^{n_{i}},x^{-n_{i}}])(\ol\gamma^i)$$
and $S$ a $*$-algebra of the form 
$\bigoplus_{i\in I} \M_{\mu_i} (K[x,x^{-1}])$
where $I$ is a set, $\mu_i$ cardinals, $n_i$ positive integers, and $\ol \gamma^i\in \Zset^{\mu_i}$ for $i\in I.$
The following conditions are equivalent. 
\begin{enumerate}[\upshape(1'')]
\item $L_K(E)$ is graded $*$-isomorphic to $S^{\gr}.$
\item $L_K(E)$ is graded isomorphic to $S^{\gr}.$
\item $L_K(E)$ is $*$-isomorphic to $S.$
\item $L_K(E)$ is isomorphic to $S.$
\item $E$ is a row-finite, no-exit graph without sinks such that every infinite path ends in a cycle.
\end{enumerate}  
\label{comet_corollary}
\end{corollary}
\begin{proof}
The implications (1'') $\Rightarrow$ (2'') $\Rightarrow$ (4'') and (1'') $\Rightarrow$ (3'') $\Rightarrow$ (4'') are direct. Since the socle of the algebra $K[x,x^{-1}]$ is zero, (4'') implies that the socle of $L_K(E)$ is zero as well. Thus, there are no sinks in $L_K(E)$ by \cite[Corollary 5.3]{Gonzalo_et_al_socle}. 
Also, $E$ is row-finite and its every infinite path ends in a cycle by Corollary \ref{no-exit_corollary} so (5'') holds. Condition (5'') implies condition (1) of Corollary \ref{no-exit_corollary}. The absence of sinks (and regular vertices not in set $T$ for any complete subobject $(F,T)$ of $(E, R(E))$ since $E$ is row-finite) implies the absence of the matrix algebras over $K$ in the representation from Corollary \ref{no-exit_corollary}. Thus, condition (1'') holds. 
\end{proof}

\subsection{Realization Question.}
The reader may wonder whether every graded locally matricial $*$-algebra of the form as $R^{\gr}$ in Corollary \ref{no-exit_corollary} can be realized as a Leavitt path algebra of some no-exit graph and we refer to this question as the Realization Question. In fact, the referee of this paper had the same question. Referee's question inspired the author to include the following proposition in the paper. The proposition shows that the Realization Question has a negative answer. 
\begin{proposition}
Let $K$ be a field trivially graded by the set $\Zset$ of integers and $R$ be the algebra $\M_2(K)(0,0).$ The algebra $R$ is not graded $*$-isomorphic to any Leavitt path algebra. 
\label{realization_question} 
\end{proposition}
\begin{proof}
Assume that there is a graph $E$ and such that $R$ and $L_K(E)$ are graded $*$-isomorphic. By Corollary \ref{acyclic_corollary}, $E$ is a row-finite, acyclic graph such that every infinite path ends in a sink. Moreover, since $R$ is unital, $E^0$ is finite and, since $E$ is row-finite with $E^0$ finite, $E^1$ is finite also.  Since $R$ is simple, $E$ has just one sink (otherwise $L_K(E)$ would be isomorphic to a direct sum of more than one matricial algebra and, thus, would not be simple). Let $Q$ be a graded locally matricial representation of $L_K(E).$
By  Corollary \ref{acyclic_corollary}, $Q$ is graded $*$-isomorphic to single graded matrix algebra. Since $\M_n(K)\cong \M_2(K)$ implies that $n=2,$ there are two paths ending in the sink of $E$ by Proposition \ref{ultramatricial_rep}. One path is the trivial path of length 0. If $k$ is the length of the other path, then $k\leq 1$ since otherwise there would be more than two paths ending at the sink of $E$. Thus $k=0$ or $k=1.$ However, $k=0$ cannot happen since there cannot be two different paths of length zero ending in a vertex. If $k=1,$ the algebras $R=\M_2(K)(0, 0)$ and $Q\cong_{\gr}\M_2(K)(0, 1)$ are not graded isomorphic since their zero-components,
$\left[
\begin{array}{cc}
K & K\\ 
K & K\\
\end{array}
\right]$ and $\left[
\begin{array}{cc}
K & 0\\ 
0 & K\\
\end{array}
\right]$ respectively, are not isomorphic.   
\end{proof}

\subsection{The monoid of finitely generated projective modules of a Leavitt path algebra.}

In this section we characterize when the monoid of finitely generated projective modules of a Leavitt path algebra is atomic and cancellative. We let  
$V(R)$ denote the monoid of isomorphism classes of finitely generated projective modules over $R$ or, equivalently, the monoid of the equivalence classes of conjugated idempotent matrices. We let $[x]$ denote the equivalence class of an idempotent matrix $x$. If $R$ is a ring, possibly without an identity element, $V(R)$ is defined using unitization $R^u$ of $R$ (\cite[\S 10]{Ara_Goodearl} has more details).

By \cite[Theorem 2.4]{AAPM}, if $E$ is a row-finite and countable graph, then the following conditions are equivalent. 
\begin{enumerate}[1.]
\item $E$ is acyclic and every infinite path ends in a sink. 
\item $L_K(E)$ is regular and $V(L_K(E))$ is cancellative and atomic.
\item $L_K(E)$ is regular and $V(L_K(E))$ is isomorphic to a direct sum of copies of the monoid of nonnegative integers $\Nset.$ 
\end{enumerate}
We shall delete the assumptions on the cardinality of the graph and characterize when $V(L_K(E))$ is atomic and cancellative (without the condition that $L_K(E)$ is regular). 

\begin{proposition}
Let $E$ be a graph and $K$ a field. 
\begin{enumerate}[(i)]
\item If $V(L_K(E))$ is cancellative, then $E$ is no-exit. 

\item If $V(L_K(E))$ is atomic, then every infinite path of $E$ ends in a sink or a cycle. 

\item If $V(L_K(E))$ is atomic, then $E$ is row-finite. 
\end{enumerate}
\label{projectives_necessary}
\end{proposition}
The statements (i) and (ii) have been shown in the proof of \cite[Theorem 2.4]{AAPM} without using the assumptions on the cardinality of the graph. We outline the proofs for completeness. 
\begin{proof}
(i) Assume that $V(L_K(E))$ is cancellative and that $E$ has a cycle $c$ with an exit $e.$ Base the cycle $c$ at $v=\so(e)$ and let $p_n=c^n(c^*)^n.$ The idempotents $p_n$ are such that $L_K(E)p_{n+1}\subsetneq L_K(E)p_n$ and that $[v]=[p_n]$ for any nonnegative integer $n.$ The relation $[p_n]=[p_{n+1}]+[p_n-p_{n+1}]$ implies that $[v]=[v]+[p_n-p_{n+1}]$ and, by the assumption that the monoid is cancellative, that $[p_n-p_{n+1}]=0.$ This is a contradiction since then $p_n-p_{n+1}=v(p_n-p_{n+1})=0$ implies that $L_K(E)p_{n+1}=L_K(E)p_n.$

(ii) Assume that $V(L_K(E))$ is atomic and that $E$ has an infinite path $p$ which does not end in a sink or a cycle. The path $p$ necessarily has infinitely many bifurcations and we can write $p$ as $q_1q_2\ldots$ such that the range of each $q_n$ is a bifurcation and let $p_n=q_1\ldots q_n(q_1\ldots q_n)^*.$ If $v=\so(p),$ the idempotents $p_n$ are such that $vL_K(E)vp_{n+1}\subsetneq vL_K(E)vp_n,$ that $[\ra(q_n)]=[p_n]$ for any positive integer $n,$ and that 
\[[p_1]=[p_1-p_2]+[p_2-p_3]+\ldots+[p_{n-1}-p_n]+ [p_n].\]

By \cite[Theorem 3.5 and Proposition 4.4]{Ara_Moreno_Pardo}, $V(L_K(E))$ is a refinement monoid if $E$ is a row-finite graph. By \cite[Corollary 5.16]{Ara_Goodearl}, the same statement holds for arbitrary graphs. Since $V(L_K(E))$ is atomic by assumption, write $[p_1]$ as a sum of atoms and, if this sum has $m$ terms, consider $n> m.$ The fact that $V(L_K(E))$ is an atomic, refinement monoid and $n>m$ implies that at least one of the terms $[p_1-p_2], [p_2-p_3], \ldots, [p_{n-1}-p_n]$ or $ [p_n]$ is zero. This lead to a contradiction analogously as in the proof of (i).

(iii) Assume that $V(L_K(E))$ is atomic and that $E$ has an infinite emitter $v.$ If $\{e_n \mid n=1,2\ldots\}$ is a set of different edges in $\so^{-1}(v),$ consider the orthogonal idempotents $e_ne_n^*$ and nonzero idempotents $p_n=v-\sum_{i=1}^n e_ie_i^*$ orthogonal to every $e_ie_i^*,$ $i=1,\ldots,n.$ The relation $v=\sum_{i=1}^n e_ie_i^*+p_n$ implies that \[[v]=\sum_{i=1}^n[e_ie_i^*]+[p_n].\]

Let $a_1,\ldots, a_m$ be atoms such that $[v]=a_1+\ldots+a_m.$ Taking any $n>m,$ and using that  $V(L_K(E))$ is a refinement monoid, we obtain that either 
$[e_ie_i^*]=0$ for some $i=1,\ldots, n$ or $[p_n]=0$. The first case, $[e_ie_i^*]=0,$ implies that $L_K(E)e_ie_i^*=0$ which is a contradiction since $e_ie_i^*=ve_ie_i^*=0$ implies that $e_i=e_ie_i^*e_i=0$ and $e_i$ is a basis element. The second case, $[p_n]=0,$ implies that $L_K(E)p_n=0$ which is a contradiction since $p_n=vp_n=0$ implies that $e_{n+1}=ve_{n+1}=\sum_{i=1}^n e_ie_i^*e_{n+1}=0$ and $e_{n+1}$ is a basis element.   
\end{proof}

\begin{corollary}
Let $E$ be a graph, $K$ a field, and let $\Nset$ stand for the monoid of nonnegative integers. The following conditions are equivalent. 
\begin{enumerate}
\item[(5)] $E$ is a row-finite, no-exit graph in which every infinite path ends in a sink or a cycle. 

\item[(6)] $V(L_K(E))$ is isomorphic to a direct sum of copies of $\Nset.$ 

\item[(7)] $V(L_K(E))$ is atomic and cancellative.
\end{enumerate}
\label{projectives} 
\end{corollary}
The numbering of the first condition refers to the numbering in Corollary \ref{no-exit_corollary}. 
\begin{proof}
Condition (5) implies that $L_K(E)$ is isomorphic to an algebra of the form 
$$\bigoplus_{i\in I} \M_{\kappa_i} (K) \oplus \bigoplus_{j \in J} \M_{\mu_j} (K[x,x^{-1}])$$
by Corollary \ref{no-exit_corollary}. Thus, the monoid $V(L_K(E))$ is isomorphic to the monoid $\Nset^{|I|} \oplus \Nset^{|J|}$ which proves (6). The implication (6) $\Rightarrow$ (7) is direct. Condition (7) implies (5) by Proposition \ref{projectives_necessary}.
\end{proof}

\section{Locally noetherian Leavitt path algebras}\label{section_noetherian}

The property that a ring is locally noetherian generalizes the property that ring is noetherian for locally unital rings. In this section, we introduce 
the graded version of the property that a ring is locally noetherian. The main result, Theorem \ref{noetherian}, characterizes Leavitt path algebras which are graded locally noetherian. As a corollary, the assumptions that the underlying graph is countable and row-finite can be dropped from the characterization of locally noetherian Leavitt path algebras (\cite[Theorem 3.7]{AAPM}).

\subsection{Graded locally and categorically noetherian rings} Recall that a ring $R$ is said to be {\em locally unital} if for every finite set $F$ of elements of $R$, there is an idempotent $u$ such that $ux=xu=x$ for every $x\in F.$ The set of all such idempotents $u$ is a set of local units. 

Some authors (e.g. \cite{Tomforde}) define a ring $R$ to be locally unital by a stronger condition: there is a set of commuting idempotents $u_i, i\in I,$ such that for every $x\in R,$ there is $i\in I$ with $xu_i=u_ix=x.$ In this case, the set of local units is a directed set (the idempotent $u_i+u_j-u_iu_j$ is an upper bound for $u_i$ and $u_j$) and $R$ is the direct limit of rings $u_iRu_i.$ If the idempotents $u_i$ are pairwise orthogonal, then $R=\bigoplus_{i\in I} u_iRu_i.$ 

A ring $R$ is said to have {\em enough idempotents} if there is a set of pairwise orthogonal idempotents $u_i, i\in I,$ such that $R=\bigoplus_{i\in I} Ru_i=\bigoplus_{i\in I} u_iR.$ Every ring with enough idempotents is locally unital since the finite sums of the idempotents $u_i$ are local units. If $R$ has an orthogonal set of local units, then it has enough idempotents too. 

If $R$ is a ring with enough idempotents or with local units, then one usually considers unitary left $R$-modules (a left module $M$ is unitary, or full, if $RM = M$, see \cite{Garcia_Simon}, \cite[\S 10]{Ara_Goodearl} or \cite[\S 1]{AAPM}). If $R$ has enough idempotents, then it is a unitary left and right module. As a consequence, a free left module over such ring is also unitary and, hence, any quotient of a free left module is unitary as well.

We adapt these concepts to graded rings. Namely, a ring-theoretic property can be adapted to graded rings by considering {\em homogeneous} elements instead of arbitrary elements in the definition of the property. For example, a ring is regular if $x\in xRx$ for every element $x.$ Thus, a graded ring $R$ is said to be {\em graded} regular if $x\in xRx$ for every {\em homogeneous} element $x$. Following this principle, we  define a graded locally unital ring, a graded locally noetherian ring and, in the next section, a graded locally artinian ring. 

\begin{definition}\label{definition_of_graded_locally_unital}
A graded ring $R$ is {\em graded locally unital} if for every finite set $F$ of (homogeneous) elements of $R$, there is a homogeneous idempotent $u$ such that $ux=xu=x$ for every $x\in F.$   

A graded ring $R$ has {\em enough homogeneous idempotents} if there is a set of pairwise orthogonal homogeneous idempotents $u_i, i\in I,$ such that $R=\bigoplus_{i\in I} Ru_i=\bigoplus_{i\in I} u_iR.$ 
\end{definition}

The word ``homogeneous'' appears in parenthesis in one instance in the definition above because requiring that the elements of the set $F$ are homogeneous is equivalent to requiring that the elements of the set $F$ are arbitrary.  Indeed, if it holds that for every finite set $F_h$ of homogeneous elements of $R$, there is a homogeneous idempotent $u$ such that $ux=xu=x$ for every $x\in F_h,$ and $F$ is any finite set of  elements of $R$, then let $F_h$ be the set of all homogeneous elements of $R$ which appear in representations of elements of $F$ as sums of homogeneous elements. Let $u$ be a homogeneous idempotent such that $ux=xu=x$ for each $x\in F_h.$ Then clearly $ux=xu=x$ for each $x\in F$ too. 

Analogously to the non-graded case, every graded ring with enough homogeneous idempotents is graded locally unital. If a graded ring $R$ has an orthogonal set of graded local units, then $R$ has enough homogeneous idempotents. 

A ring of the form $R=\bigoplus_{i\in I} Ru_i=\bigoplus_{i\in I} u_iR$ or $R=\bigoplus_{i\in I} u_iRu_i$ where $u_i$ are orthogonal idempotents of $R$ for $i\in I,$ is clearly neither left nor right noetherian. However, it is still of relevance to know if such ring is build up from left or right noetherian pieces. This is precisely the reason that motivated the introduction of locally and categorically noetherian rings. We recall these definitions first and then we adapt them to graded rings. 

A ring $R$ is said to be {\em locally left noetherian} if for every finite set $F$ of elements of $R$, there is an idempotent $e\in R$ such that $eRe$ contains $F$ and
$eRe$ is left noetherian. A locally right noetherian ring is defined analogously. A ring is locally noetherian if it is both locally left and locally right
noetherian. 

A ring is {\em categorically  left noetherian} if every finitely generated left module is noetherian. A categorically right noetherian ring is defined analogously. A ring is categorically noetherian if it is both categorically left and categorically right noetherian. 

By \cite[Proposition 1.2]{AAPM}, if $R=\bigoplus_{i\in I} Ru_i$ for a set of orthogonal idempotents $u_i, i\in I,$ then $R$ is categorically left noetherian if and only if $Ru_i$ is noetherian left module for each $i\in I.$ Thus, every ring $R$ with enough idempotents $u_i$ is categorically left noetherian if and only if $Ru_i$ is noetherian for each $i\in I.$ Such $R$ is categorically noetherian if and only if $Ru_i$ and $u_iR$ are noetherian for each $i\in I.$

A ring $R$ is locally left noetherian if and only if there are local units $u_i, i\in I,$ such that $u_iRu_i$ is left noetherian for every $i$ (see \cite[page 99]{AAPM}). Moreover, if $R$ is locally left noetherian and $v_j,j\in J,$ is any set of local units, then $v_jRv_j$ is left noetherian. This statement can be proven analogously as part (4) of Lemma \ref{noetherian_properties} below, which is the graded version of this fact.  

By \cite[Lemma 1.6]{AAPM}, if $R$ is a categorically left  noetherian ring with local units, then $R$ is locally left  noetherian too. By the paragraph preceding \cite[Lemma 1.6]{AAPM}, not every locally unital and locally noetherian ring is categorically noetherian. 

We turn to the graded rings now. A graded module $M$ over a graded ring $R$ is {\em graded noetherian} if every ascending chain of graded submodules of $M$ is constant eventually (equivalently, if every graded submodule of $M$ is finitely generated). A graded ring $R$ is {\em graded left noetherian} if it is graded noetherian as a left module (equivalently, if every graded left ideal is finitely generated). A graded right noetherian ring is defined analogously and a graded noetherian ring is defined using the usual convention.

If $R$ is a commutative ring graded by a finitely generated abelian group $\Gamma,$ then the following conditions are equivalent by \cite[Theorem 1.1]{Goto_Yamagishi}.  
\begin{enumerate}[1.]
\item $R$ is noetherian. 

\item $R$ is graded noetherian.

\item $R_0$ is noetherian and $R$ is finitely generated as an $R_0$-algebra.  
\end{enumerate}

The above equivalences do not hold if the group $\Gamma$ is not finitely generated (see \cite[Example 3]{Goto_Yamagishi}). In addition, the example at the end of \cite[\S 1.1.4]{Roozbeh_graded_ring_notes} exhibits a commutative ring which is graded noetherian but not noetherian. By \cite[Theorem 5.4.7]{Graded_Rings_book},  conditions (1) and (2) are equivalent for any $\Zset$-graded ring (not necessarily commutative). Thus, conditions (1) and (2) are equivalent for Leavitt path algebras. Theorem \ref{noetherian} implies analogous conclusion when ``noetherian'' is replaced by locally or categorically noetherian.  

We define the graded versions of locally and categorically noetherian rings now. 

\begin{definition}
A graded ring $R$ is {\em graded locally left noetherian} if for every finite
set $F$ of (homogeneous) elements of $R$, there is a homogeneous idempotent $e\in R_0$ such that $eRe$ contains $F$ and
$eRe$ is graded left noetherian. 

A graded ring is {\em graded categorically left noetherian} if every finitely generated graded left module is graded noetherian. 
\end{definition}
Requiring that the set $F$ consists of homogeneous elements in the definition of graded locally unital ring is equivalent to requiring that the set $F$ consists of arbitrary elements. This fact can be shown using the same argument as in the non-graded case in the paragraph following Definition \ref{definition_of_graded_locally_unital}. 

A graded locally right noetherian ring and a graded categorically right noetherian ring are defined analogously and a graded locally noetherian ring and a graded categorically noetherian ring are defined using the usual conventions. 

We prove some basic properties of the introduced concepts in the next lemma.

\begin{lemma}\label{noetherian_properties}
Let $R$ be a graded ring. 
\begin{enumerate}
\item If $R=\bigoplus_{i\in I} Ru_i$ for a set of homogeneous idempotents $u_i, i\in I,$ then $R$ is graded categorically left noetherian if and only if $Ru_i$ is graded noetherian left module for each $i\in I.$ This statement also has its right-sided analogue. 

\item If $R$ is categorically left  noetherian, then $R$ is graded categorically left noetherian. 

\item The ring $R$ is graded locally left noetherian if and only if $R$ has homogeneous local units $u_i, i\in I,$ such that $u_iRu_i$ is graded left noetherian for every $i.$ 

\item If $R$ is graded locally left  noetherian and $u_i,i\in I,$ is any set of homogeneous local units, then $u_iRu_i$ is graded left  noetherian. 

\item If $R$ is graded locally unital and locally left  noetherian, then $R$ is graded locally left  noetherian. 

\item If $R$ is graded locally unital and graded categorically left  noetherian, then it is graded locally left noetherian. 
\end{enumerate}
The statements (2) to (6) have their analogues with the words ``left'' replaced by ``right''.  
\end{lemma}
\begin{proof}
To show the direction $(\Rightarrow)$ of (1), note that if $u_i$ is homogeneous, then $Ru_i$ is graded. Thus, if $R$ is graded categorically left noetherian, then $Ru_i$ is graded noetherian. For the converse, if $Ru_i$ is graded noetherian for every $i\in I$, then any finitely generated graded free left module is a submodule of a direct sum of finitely many modules $Ru_i$ and, hence, graded noetherian. As a consequence, any graded left module with finitely many homogeneous generators (which is a graded homomorphic image of a finitely generated graded free left module, see \cite[Section 1.2.4]{Roozbeh_graded_ring_notes}) is graded noetherian too. 

If $R$ is graded, then every noetherian module over $R$ is also graded noetherian. Thus (2) holds. 

To show (3), note that if $R$ is graded locally left noetherian, then the set of all homogeneous idempotents $e$ such that $eRe$ is graded left noetherian is a set of graded local units. Conversely, if $u_i, i\in I$ is a set of graded local units such that $u_iRu_i$ is graded left noetherian for all $i\in I,$ and $F$ is a finite set of homogeneous elements, then there is a local unit $u_i$ such that $F\subseteq u_iRu_i.$ This implies that $R$ is graded locally left noetherian by definition. 

To show (4), let $R$ be a graded locally left noetherian ring. Then, there is a set of homogeneous local units $e_j, j\in J,$ such that $e_jRe_j$ is left noetherian. If $R$ also has homogeneous local units $u_i, i\in I,$ then for every $i\in I,$ there is $j\in J$ such that $u_iRu_i\subseteq e_jRe_j$. If $I_n, n=1,2,\ldots,$ is an increasing sequence of graded left $u_iRu_i$-ideals, then $e_jRe_jI_n$ is an increasing sequence of graded left $e_jRe_j$-ideals. Thus, this sequence is constant and there is $n$ such that  $e_jRe_jI_n= e_jRe_jI_m$ for all $m\geq n.$ Hence $u_iRu_ie_jRe_jI_n=u_iRu_ie_jRe_jI_m$ for all $m\geq n$ as well. Since $u_iRu_ie_jRe_jI_m=u_iRu_iRe_jI_m=u_iRu_iRe_ju_iRu_iI_m=u_iRu_iI_m=I_m$ for all $m\geq n,$ we have that $I_n=I_m$ for all $m\geq n.$ Thus, $u_iRu_i$ is graded left noetherian.  

The proof of (4) also shows (5) since every increasing sequence of graded ideals is also an increasing sequence of ideals. 

The proof of (6) is analogous to the non-graded case (\cite[Lemma 1.6]{AAPM}). Let $R$ be graded categorically left noetherian with a set of graded local units $u_i, i\in I.$ By (3), it is sufficient to show that $u_iRu_i$ is graded left noetherian for any $i\in I.$ Let $I_n, n=1,2,\ldots,$ be an increasing sequence of left $u_iRu_i$-ideals for $i\in I$. Since $Ru_i$ is a finitely generated graded left ideal, it is graded noetherian. Thus, the increasing sequence $RI_n$ is constant eventually and so the sequence $u_iRI_n$ is constant eventually also. Since $u_iRI_n=u_iRu_iI_n=I_n,$ for any $n,$ the sequence $I_n$ is constant as well.  
\end{proof}

\subsection{Locally noetherian Leavitt path algebras}

By \cite[Theorem 3.7]{AAPM}, if $E$ is a countable, row-finite graph and $K$ a field, the following conditions are equivalent. 
\begin{enumerate}[1.] 
\item $L_K(E)$ is categorically left (right) noetherian. 
\item $L_K(E)$ is locally left (right) noetherian. 
\item $E$ is a no-exit graph such that every infinite path ends in a sink or a cycle. 
\item $L_K(E) \cong \bigoplus_{i\in I} \M_{\kappa_i} (K) \oplus \bigoplus_{j \in J} \M_{\mu_j} (K[x,x^{-1}]),$ where $I$ and $J$ are countable sets, and $\kappa_i, \mu_j$ are countable cardinals.
\end{enumerate}
The isomorphism in the last condition can be taken to be a ring or an algebra isomorphism. In \cite[Corollary 32]{Zak_Lia}, it is shown that the assumption that $E$ is countable can be dropped (in which case $I,$ $J,$ $\kappa_i,$ and $\mu_j$ may be of arbitrary cardinalities) for the equivalence of the last two conditions and that the isomorphism in the last condition can be taken to be a $*$-isomorphism. Our next result has no assumptions on the cardinality of the graph and incorporates the graded versions of the ring-theoretic properties as well. 

\begin{theorem}\label{noetherian}
For any graph $E$ and a field $K,$ the following conditions are equivalent. 
\begin{enumerate}
\item[(5)]  $E$ is a row-finite, no-exit graph such that every infinite path ends in a sink or a cycle. 
\item[(8l)] $L_K(E)$ is categorically left noetherian. 
\item[(8r)] $L_K(E)$ is categorically right noetherian. 
\item[(9l)] $L_K(E)$ is locally left noetherian.
\item[(9r)] $L_K(E)$ is locally right noetherian.
\item[(10l)] $L_K(E)$ is graded categorically left  noetherian. 
\item[(10r)] $L_K(E)$ is graded categorically right noetherian. 
\item[(11l)] $L_K(E)$ is graded locally left noetherian.
\item[(11r)] $L_K(E)$ is graded locally right noetherian.

\end{enumerate}
We let (8), (9), (10) and (11) denote conditions (8l), (9l), (10l) and (11l) respectively, with the word ``left'' deleted. 
\end{theorem} 
\begin{proof}
The left-right symmetry in conditions (8l) to (11r) holds since the vertices are homogeneous local units for $L_K(E)$ making $L_K(E)$ a ring with enough homogeneous idempotents such that $(L_K(E)v)^*=vL_K(E)$ and $(vL_K(E)v)^*=vL_K(E)v$ for any vertex $v.$

The implications  (8) $\Rightarrow$ (9), (9) $\Rightarrow$ (11), (8) $\Rightarrow$ (10) and (10) $\Rightarrow$ (11) hold by parts (2), (5) and (6) of Lemma \ref{noetherian_properties} since the vertices are orthogonal, homogeneous local units for $L_K(E).$  

The proof of \cite[Theorem 3.7]{AAPM} shows that (9) implies that $E$ is a no-exit graph in which infinite paths end in sinks or cycles without using that the graph $E$ is countable or row-finite. In particular, assuming that there is a cycle $c$ with an exit based at a vertex $v$, the left $vL_K(E)v$-ideals $vL_K(E)v(v-c^n(c^*)^n)$ form a strictly increasing chain. Since these ideals are generated by homogeneous elements, they are graded. Also, if there is an infinite path $p=q_1q_2\ldots$ which does not end in a sink or a cycle and has bifurcations at $\ra(q_n)$ for every $n,$ let $v=\so(p)$ and $p_n=q_1\ldots q_n(q_1\ldots q_n)^*.$ The left $vL_K(E)v$-ideals $vL_K(E)v(v-p_n)$ form a strictly increasing chain. These ideals are also generated by homogeneous elements and so they are graded. Thus, (11) implies that $E$ is a no-exit graph in which infinite paths end in sinks or cycles. So, for (11) $\Rightarrow$ (5), it is sufficient to show that (11) implies that $E$ is row-finite. Let us assume that (11) holds and that $E$ has an infinite emitter $v.$ Let $\{e_n \mid n=1,2\ldots\}$ be a set of different edges in $\so^{-1}(v)$ and $p_n=\sum_{i=1}^n e_ie_i^*.$ Then the graded left $vL_K(E)v$-ideals $vL_K(E)vp_n$ constitute a strictly increasing chain of graded left ideals. This contradicts (11). Hence $E$ is row-finite.   

To finish the proof, it is sufficient to show that (5) implies (8). By Corollary \ref{no-exit_corollary}, it is sufficient to note that a direct sum of algebras each of which is of the form $\M_\kappa(K)$ or $\M_\kappa(K[x, x^{-1}])$ for a cardinal $\kappa,$ is categorically noetherian. This can be checked directly by definition.   
\end{proof}

\begin{remarks}
\label{properties_12_to_15} 
\begin{enumerate}
\item[(i)] Results of \cite{Roozbeh_Ranga} and \cite{Roozbeh_Lia_Baer} imply that the following conditions are equivalent with conditions from Theorem \ref{noetherian}.
\begin{enumerate}
\item[(12)] $L_K(E)$ is (graded) locally Baer (i.e. every (graded) idempotent-generated corner of $L_K(E)$ is Baer, see \cite[Theorem 15]{Roozbeh_Lia_Baer}). 

\item[(13)] $L_K(E)$ is graded self-injective (see \cite[Theorem 6.7]{Roozbeh_Ranga}).

\item[(14)] $L_K(E)$ coincides with its graded socle (see \cite[Theorem 2.10 and Remark p. 469]{Roozbeh_Ranga}).

\item[(15)] $L_K(E)$ is graded semisimple (i.e. $L_K(E)$ is a direct sum of minimal graded ideals, see \cite[Theorems 2.10, 4.5 and Remark p. 469]{Roozbeh_Ranga}). 
\end{enumerate}

\item[(ii)] 
By Theorem \ref{noetherian}, a Leavitt path algebra is locally (or categorically) noetherian if and only if it is graded locally (or categorically) noetherian. Thus, 
the property of being locally noetherian is invariant for the graded structure of Leavitt path algebras. This happens also with the property of being locally Baer 
by \cite[Theorem 15]{Roozbeh_Lia_Baer}, but does not happen with every ring-theoretic property. For example, every Leavitt path algebra is graded regular by \cite[Theorem 9]{Roozbeh_regular}, while just Leavitt path algebras of acyclic graphs are regular as non-graded rings by \cite[Theorem 1]{Gene_Ranga}. Also, the classes of Leavitt path algebras which are locally Baer $*$-rings and which are graded locally Baer $*$-rings are different by \cite[Theorems 15 and 16]{Roozbeh_Lia_Baer}.
\end{enumerate}
\end{remarks}

\section{Locally artinian Leavitt path algebras}\label{section_artinian}

In this section, we introduce the graded versions of the properties of being categorically and locally artinian. The main result, Theorem \ref{artinian}, characterizes Leavitt path algebras which are graded locally and graded categorically artinian. As a corollary, the assumptions that the underlying graph is countable and row-finite can be dropped from the characterization of locally artinian Leavitt path algebras in \cite[Theorem 2.4]{AAPM}. Also, it is interesting to point out that our results imply that the property of being noetherian and artinian differ for Leavitt path algebras in the following sense.  
\begin{center}
\begin{tabular}{l}
{\em $L_K(E)$ is graded locally noetherian if and only if $L_K(E)$ is locally noetherian}\hskip.4cm while \\
{\em $L_K(E)$ can be graded locally artinian without being locally artinian.}
\end{tabular}
\end{center}

\subsection{Graded locally and categorically artinian rings}
A ring $R$ is said to be {\em locally left artinian} if for every finite set $F$ of elements of $R$, there is an idempotent $e\in R$ such that $eRe$ contains $F$ and
$eRe$ is left artinian. A ring is {\em categorically left artinian} if every finitely generated left module is artinian. A locally right artinian ring and a categorically right artinian ring are defined analogously. A locally artinian ring and a categorically artinian ring are defined using the usual conventions.

By \cite[Proposition 1.2]{AAPM}, if $R=\bigoplus_{i\in I} Ru_i$ for a set of idempotents $u_i, i\in I,$ then $R$ is categorically left artinian if and only if $Ru_i$ is artinian left module for each $i\in I.$ A ring $R$ is locally left artinian if and only if there are local units $u_i, i\in I,$ such that $u_iRu_i$ is left artinian for every $i$ (see \cite[page 99]{AAPM}). By \cite[Lemma 1.5]{AAPM}, if $R$ is a categorically left artinian ring with local units, then $R$ is locally left artinian too. By the paragraph preceding \cite[Lemma 1.5]{AAPM}, not every locally unital and locally artinian ring is categorically artinian. 

The definitions above adapt to graded rings as follows. A graded module over a graded ring $R$ is {\em graded artinian} if every descending chain of graded submodules is constant eventually. A graded ring $R$ is {\em graded left artinian} if it is graded artinian as a left module. A graded right artinian ring is defined analogously and a graded artinian ring is defined using the usual convention. 

\begin{definition}
A graded ring $R$ is {\em graded locally left artinian} if for every finite set $F$ of (homogeneous) elements of $R$, there is a homogeneous idempotent $e$ such that $eRe$ contains $F$ and $eRe$ is graded left artinian. 

A graded ring is {\em graded categorically left artinian} if every finitely generated graded left module is graded artinian.
\end{definition}
Similarly as in the definition of a graded locally left noetherian ring, the statements with and without the word ``homogeneous'' in parenthesis in the definition of a graded locally left artinian ring are equivalent. A graded locally right artinian ring and a graded categorically right artinian ring are defined analogously and a graded locally artinian ring and a graded categorically artinian ring are defined using the usual conventions. 

We summarize some of the properties of the artinian-related concepts in the next lemma. The proof of the lemma is completely analogous to the proof of Lemma \ref{noetherian_properties} and we omit it. 

\begin{lemma}\label{artinian_properties} 
Let $R$ be a graded ring. 
\begin{enumerate}
\item If $R=\bigoplus_{i\in I} Ru_i$ for a set of homogeneous idempotents $u_i, i\in I,$ then $R$ is graded categorically left artinian if and only if $Ru_i$ is graded artinian left module for each $i\in I.$ This statement also has its right-sided analogue. 

\item If $R$ is categorically left  artinian, then $R$ is graded categorically left  artinian. 

\item The ring $R$ is graded locally left  artinian if and only if $R$ has homogeneous local units $u_i, i\in I,$ such that $u_iRu_i$ is graded left artinian for every $i.$ 

\item If $R$ is graded locally left  artinian and $u_i,i\in I,$ is any set of homogeneous local units, then $u_iRu_i$ is graded left  artinian. 

\item If $R$ is graded local unital and locally left artinian, then $R$ is graded locally left  artinian. 

\item If $R$ is graded locally unital and graded categorically left  artinian, then it is graded locally left artinian.  
\end{enumerate}
The statements (2) to (6) have their analogues with the words ``left'' replaced by ``right''. 
\end{lemma}

In \cite{Roozbeh_Ranga}, a ring $R$ is said to be {\em graded semiprime} if $R$ has no nonzero nilpotent graded ideals. If $R$ is graded locally unital, this condition is equivalent with $xRx=0$ implies that $x=0$ for every homogeneous element $x.$ Using this characterization, it is direct to check that every graded locally unital and graded regular ring is graded semiprime.  

A graded semiprime ring $R$ is said to be {\em graded semisimple} if $R$ is a direct sum of minimal graded left (equiv. right) ideals. This last condition is equivalent to $R$ being equal to its graded socle. The following lemma states that some well-known facts for non-graded rings hold for graded rings as well. 

\begin{lemma}\label{artinian_vs_semisimple} 
Let $R$ be a graded ring. 
\begin{enumerate}[1.] 
\item If $R$ is graded semisimple, then it is graded categorically artinian. 
\item If $R$ is unital, graded regular and graded artinian, then $R$ is graded semisimple. 
\item If $R$ is graded regular and graded locally artinian, then $R$ is graded semisimple. 
\end{enumerate}
\end{lemma}
\begin{proof}
To show (1), let $R$ be graded semisimple. Each minimal graded left ideal from the decomposition of $R$ as a direct sum of such ideals, is graded artinian since it is graded simple. Thus, any finitely generated graded free left module is graded artinian and so any finitely generated graded left module (which is a graded homomorphic image of a finitely generated graded free left module) is graded artinian too. 

The proof of (2) is analogous to the non-graded version. If $R$ is unital and graded regular, then the graded Jacobson radical is zero (by the graded version of \cite[Corollary 1.2]{Goodearl_book}). If $R$ is also graded artinian, then the graded version of Artin-Wedderburn theorem (see \cite[Remark 1.4.8]{Roozbeh_graded_ring_notes}) implies that $R$ is a finite direct sum of minimal graded left ideals each of which is graded isomorphic to a graded matrix ring over a graded field (a commutative graded ring in which every nonzero homogeneous element has a multiplicative inverse). Thus, $R$ is graded semisimple.  

To show (3), it is sufficient to show that $R$ is equal to its graded socle. We recall that the local ring $R_a$ for $a\in R$ is defined as $R_a=\{axa\, |\, x\in R
\}$ with multiplication $axa\cdot aya=axaya$. The element $axa$ is denoted by $\overline{x}$ when considered as an element of $R_a.$ We also recall 
\cite[Proposition 2.1 (v)]{Mercedes_Lozano} stating that for an element $a$ of a semiprime ring $R,$ $\overline{x}$ is in the socle of $R_a$ if and only if $axa$ is in the socle of $R$. It is direct to check that the graded version of this statement holds as well. In the special case when $a$ is a homogeneous idempotent and $x$ is a homogeneous element of $aRa,$ then $x$ is in the graded socle of $aRa$ if and only if $x$ is in the graded socle of $R$.  

Let $x\in R$ be homogeneous and $u$ a homogeneous idempotent such that $x\in uRu$ and $uRu$ is graded artinian (such $u$ exists by part (3) of Lemma \ref{artinian_properties}). Since $R$ is graded regular, $uRu$ is graded regular too and so $uRu$ is graded semisimple by part (2). Thus, $x$ is in the graded socle of $uRu$ and hence $x$ is in the graded socle of $R$ by the previous paragraph.  
\end{proof}

\subsection{Locally artinian Leavitt path algebras}

By \cite[Theorem 2.4]{AAPM}, if $E$ is a countable, row-finite graph and $K$ a field, the following conditions are equivalent. 
\begin{enumerate}[1.] 
\item $L_K(E)$ is semisimple. 
\item $L_K(E)$ is categorically left (right) artinian. 
\item $L_K(E)$ is locally  left (right) artinian. 
\item $E$ is acyclic and every infinite path ends in a sink.
\item $L_K(E) \cong \bigoplus_{i\in I} \M_{\kappa_i} (K),$ where $I$ is a countable set, and each $\kappa_i$ is a countable cardinal.
\end{enumerate}
The isomorphism in the last condition can be taken to be a ring or an algebra isomorphism.  We characterize the graded and the non-graded artinian-related properties of Leavitt path algebras without imposing any restrictions on the cardinality of the graph. 

\begin{theorem}\label{artinian}
For any graph $E$ and a field $K,$ the following conditions are equivalent with conditions (1) to (15) of Corollary \ref{no-exit_corollary}, Corollary \ref{projectives}, Theorem \ref{noetherian}, and Remark \ref{properties_12_to_15}. 
\begin{enumerate}
\item[(15)] $L_K(E)$ is graded semisimple. 
\item[(16l)] $L_K(E)$ is graded categorically left  artinian.
\item[(16r)] $L_K(E)$ is graded categorically right artinian.
\item[(17l)] $L_K(E)$ is graded locally left  artinian.
\item[(17r)] $L_K(E)$ is graded locally right artinian.
\end{enumerate}
We let (16) and (17) stand for statements (16l) and (17l) respectively, with the word ``left'' deleted. 

The following conditions are equivalent with conditions (1') to (5') of Corollary \ref{acyclic_corollary} and imply conditions (1) to (17). 
\begin{enumerate}
\item[(5')] $E$ is a row-finite, acyclic graph such that every infinite path ends in a sink.
\item[(6')] $L_K(E)$ is semisimple. 
\item[(7'l)] $L_K(E)$ is categorically left artinian. 
\item[(7'r)] $L_K(E)$ is categorically right artinian. 
\item[(8'l)] $L_K(E)$ is locally left artinian.
\item[(8'r)] $L_K(E)$ is locally right artinian.
\end{enumerate}
We let (7') and (8') stand for statements (7'l) and (8'l) respectively, with the word ``left'' deleted.  
\end{theorem}
\begin{proof}
The left-right symmetry in the conditions holds for the same reason as in Theorem \ref{noetherian}. 

The implication (15) $\Rightarrow$ (16) holds by the first part of Lemma \ref{artinian_vs_semisimple} and (17) $\Rightarrow$ (15) holds by the third part of Lemma \ref{artinian_vs_semisimple} since $L_K(E)$ is graded regular by \cite[Theorem 9]{Roozbeh_regular}. The implication (16) $\Rightarrow$ (17) holds by part (6) of Lemma \ref{artinian_properties}. Condition (1) of Corollary \ref{no-exit_corollary} implies (15) and the converse holds by \cite[Theorem 2.10]{Roozbeh_Ranga} which states that the graded socle of a Leavitt path algebra is exactly the graded algebra as in condition (1) of Corollary \ref{no-exit_corollary}. Thus, conditions (1) to (17) are equivalent. 

The implications (7') $\Rightarrow$ (16) and (8') $\Rightarrow$ (17) hold by parts (2) and (5) of Lemma \ref{artinian_properties}. 

Conditions (6') and (8') are equivalent for any semiprime ring by \cite[Theorem 2.3]{AAPM}. The implication (7') $\Rightarrow$ (8') holds by part (6) of Lemma \ref{artinian_properties} if the grading is not considered (also by \cite[Lemma 1.5]{AAPM}). The implication (5') $\Rightarrow$ (7') holds by Corollary \ref{acyclic_corollary} since any direct sum of algebras of the form $\M_\kappa(K),$ for a cardinal $\kappa,$ is categorically artinian. 

It remains to show that (8') implies (5'). Condition (8') implies condition (17), which is shown to be equivalent to condition (5) of Corollary \ref{no-exit_corollary}. Thus, $E$ is a row-finite, no-exit graph in which every infinite path ends in a sink or a cycle. Assuming that $E$ has a cycle, there is an algebra of the form $\M_{\mu_j}(K[x,x^{-1}])$ present in the algebra $S$ from Corollary \ref{no-exit_corollary}. The algebra $K[x,x^{-1}]$ is not left artinian since the left ideals $K[x,x^{-1}]x^n$ are strictly decreasing. Hence, $E$ has to be without cycles. Thus, (5') holds. 
\end{proof}

We note that the algebra $K[x,x^{-1}]$ is graded artinian since it is graded simple (see \cite[p. 463]{Roozbeh_Ranga}), but not artinian. Considering the loop $ \xymatrix{ \bullet\ar@(ru,rd)}\;\;\;\;\;$ and the fact that the Leavitt path algebra of this graph is $*$-isomorphic to $K[x,x^{-1}],$ we can see that the implication below is strict.
{\em \begin{center}
$L_K(E)$ is locally artinian $\;\;\Longrightarrow\;\;$ $L_K(E)$ is graded locally artinian 
\end{center}}
Thus, the equivalent conditions (1') to (8') of Theorem \ref{artinian} are strictly stronger than the equivalent conditions (1) to (17).

\end{document}